\documentclass[a4paper,11pt,twoside]{amsart}

\usepackage{amsfonts,amssymb,amscd, amsmath}
\usepackage{graphicx}
\usepackage[left=2.3cm,top=3.3cm,right=2.3cm,bottom=2.7cm]{geometry}
\usepackage{mathrsfs}
\usepackage{soul, todonotes, color}

\usepackage[all,ps,cmtip]{xy}
\usepackage{tikz, tikz-cd}
\usepackage{enumitem}

\usepackage[
urlcolor=blue, linktocpage, hyperindex=true, colorlinks=true,
linkcolor=blue, citecolor=blue,
]{hyperref}

\usepackage{color}

\numberwithin{equation}{section}

\DeclareRobustCommand{\SkipTocEntry}[5]{}

\def\llra{\hbox to 10mm{\rightarrowfill}}

\def\lllra{\hbox to 15mm{\rightarrowfill}}

\def\wA{{\widetilde A}}

\def\wS{{\widetilde S}}

\def\wX{{\widetilde X}}
\def\wY{{\widetilde Y}}

\def\cF{\mathcal{F}}

\def\cO{\mathcal{O}}

\def\cM{\mathcal{M}}

\def\cQ{\mathcal{Q}}

\let\tilde\widetilde

\DeclareMathOperator{\codim}{codim}
\DeclareMathOperator{\Pic}{Pic}

\DeclareMathOperator{\Exc}{Exc}

\DeclareMathOperator{\id}{id}

\DeclareMathOperator{\Aut}{Aut}

\DeclareMathOperator{\vol}{vol}

\newtheorem{lem}{Lemma}[section]

\newtheorem{thm}[lem]{Theorem}

\newtheorem{cor}[lem]{Corollary}
\newtheorem{prop}[lem]{Proposition}
\newtheorem{claim}[lem]{Claim}

\theoremstyle{definition}

\newtheorem{rema}[lem]{Remark}
\newtheorem{rmk}[lem]{Remark}

\newtheorem{qu}[lem]{Question}

\theoremstyle{remark}
\newtheorem*{remark*}{Remark}
\newtheorem*{note*}{Note}

\newcommand{\fP}{\mathfrak{P}}
\newcommand{\fV}{\mathfrak{V}}

\newcommand{\mc}{\mathcal{C}}

\newcommand{\mo}{\mathcal{O}}

\newcommand{\CC}{\mathbb{C}}

\newcommand{\PP}{\mathbb{P}}
\newcommand{\QQ}{\mathbb{Q}}
\newcommand{\RR}{\mathbb{R}}
\newcommand{\ZZ}{\mathbb{Z}}

\newcommand{\et}{\text{\'et}}

\newcommand{\Mov}{\mathrm{Mov}}

\newcommand{\rk}{\mathrm{rk}}

\begin{document}
	\title[numerically trivial automorphisms of threefolds]{On numerically trivial automorphisms of threefolds of general type}
	\author{Zhi Jiang}
	\address{Zhi Jiang \\ Shanghai Center for Mathematical Sciences \\ Xingjiangwan Campus, Fudan University \\ Shanghai 200438, China}
	\email{zhijiang@fudan.edu.cn}
	\author{Wenfei Liu}
	\address{Wenfei Liu \\ School of Mathematical Sciences \\ Xiamen University \\ Siming South Road 422, Xiamen, Fujian 361005, China}
	\email{wliu@xmu.edu.cn}
	\author{Hang Zhao}
	\address{ Hang Zhao\\Department of Mathematics \\ Southern University of Science and Technology \\ No 1088, Xueyuan Rd., Xili, Nanshan District, Shenzhen, Guangdong, China}
	\email{zhaoh3@sustech.edu.cn}
	\date{\today}
	\maketitle
	\begin{abstract}
		In this paper, we prove that the group $\Aut_\QQ(X)$ of numerically trivial automorphisms are uniformly bounded for \emph{smooth} projective threefolds $X$ of general type which either satisfy $q(X)\geq 3$ or have a Gorenstein minimal model. If $X$ is furthermore of maximal Albanese dimension, then $|\Aut_\QQ(X)|\leq 4$, and the equality can be achieved by an unbounded family of threefolds previously constructed by the third author. Along the way we prove a Noether type inequality for log canonical pairs of general type with the coefficients of the boundary divisor from a given subset $\mc\subset (0,1]$ such that $\mc\cup\{1\}$ attains the minimum.
	\end{abstract}
	\tableofcontents
	\section{Introduction}
	In studying automorphism groups of projective varieties, it is a basic problem to understand its induced action on the cohomology groups, say, with rational coefficients. More explicitly, for a complex smooth projective variety $X$, the action of the automorphism group $\Aut(X)$ on the cohomology with rational cohomology is the group homomorphism $\psi\colon \mathrm{Aut}(X)\rightarrow \Aut(H^*(X,\QQ))$ such that $\psi(\sigma) = (\sigma^{-1})^*$ for any $\sigma\in\Aut(X)$. We denote $\ker(\psi)$ by $\Aut_\QQ(X)$ and call its elements the \emph{numerically trivial automorphisms} of $X$. In this paper, we are interested in the size of $\Aut_\QQ(X)$.
	
	Obviously, the identity component $\Aut_0(X)$ of $\Aut(X)$ is contained in $\Aut_{\QQ}(X)$. It is proved by Lieberman \cite{Lie78} that the index $[\Aut_{\QQ}(X): \Aut_0(X)]$ is finite. The question is how large $[\Aut_{\QQ}(X): \Aut_0(X)]$ can be.
	
	If $\dim X=1$, then $X$ is a curve and $\Aut_{\QQ}(X)=\Aut_0(X)$ holds by the explicit description of $\Aut_0(X)$ and an easy application of the Riemann--Hurwitz formula.
	
	From dimension two on, the situation becomes very delicate. It is part of the package of the Torelli theorem that $\Aut(X)$ acts faithfully on $H^*(X, \QQ)$ (or equivalently, on $H^2(X, \QQ)$) for K3 surfaces \cite{BuR75, PS71}.\footnote{For K3 surfaces in positive characteristics, $\Aut(X)$ acts faithfully on the second $l$-adic \'etale cohomology group $H^2_{\et}(X, \QQ_l)$ by \cite[Theorem~1.4]{Ke16}.} In relation to the Torelli type problem
	in arbitrary dimension, $\Aut_\QQ(X)$ is shown to be trivial for the following classes of varieties: generalized Kummer manifolds (\cite{O20}), smooth cubic threefolds and fourfolds as well as their Fano varieties of lines (\cite{Pan20}), and complete intersections $X\subset \PP^n$ such that $\deg X\geq 3$, $\dim X\geq 2$ and $X$ is not of type $(2,2)$ (\cite{CPZ17}).
	
	However, examples of elliptic surfaces with nontrivial $\Aut_\QQ(X)$ were found around 1980 by several authors (\cite{Pe79, Pe80, BarPe83, D84, MN84}). Mukai--Namikawa \cite{MN84} obtained the optimal inequality $|\Aut_\QQ(X)|\leq 4$ for Enriques surfaces.\footnote{We refer to \cite{D13, DM19, DM21} for numerically trivial automorphisms of Enriques surfaces in positive characteristics.} By a recent work of Catanese and the second author (\cite{CatLiu20}), $[\mathrm{Aut}_{\QQ}(X):\Aut_0(X)]\leq 12$ holds for surfaces with Kodaira dimension
	0, but the index can be arbitrarily large for the class of smooth projective surfaces with Kodaira dimension $-\infty$ and $1$.
	
	If $X$ is a surface of general type, as for all smooth projective varieties of general type, the identity component $\Aut_0(X)$ is trvial, and hence $[\mathrm{Aut}_{\QQ}(X):\Aut_0(X)]= |\Aut_{\QQ}(X)|.$ Some restrictions on the invariants of $X$ with nontrivial $\Aut_{\QQ}(X)$ were found by Peters \cite{Pe80} under the condition that the canonical linear system $|K_X|$ is base point free. Cai	 \cite{Cai04} used the Bogomolov--Miyaoka--Yau inequality to show that $|\Aut_\QQ(X)|$ is uniformly bounded; in fact, $|\Aut(X)|\leq 4$ as soon as $\chi(X, \mo_X)\geq 189$. For irregular surfaces of general type, the results are quite complete by now:
	\begin{enumerate}
		\item  If $q(X)\geq 3$, then $\Aut_\QQ(X)$ is trivial (\cite{Cai12a, CLZ13}).
		\item If $q(X)=2$, then $|\Aut_\QQ(X)|\leq 2$, and there is a complete classification of the equality case: they are certain surfaces isogenous to a product of curves of unmixed type (\cite{Cai12b, CLZ13, Liu18}).
		\item If $q(X)=1$ then $|\Aut_\QQ(X)|\leq 4$, and again equality holds only if $X$ is isogenous to a product of curves of unmixed type. Moreover, there are infinite serie of such surfaces with $\Aut_\QQ(X)=4$ (\cite{CL18}).
	\end{enumerate}
	For regular surfaces of general type, it follows from \cite{Cai06} that $\Aut_\QQ(X)$ is trivial if $X$ has a genus $2$ fibration and $\chi(X, \mo_X)\geq 5$.
	
	Not much is known about $\Aut_\QQ(X)$ for higher dimensional general type varieties.  Unlike the surface case, one easily constructs smooth threefolds of general type with $\Aut_{\QQ}(X)$ nontrivial and $q(X)$ arbitrarily large: Just take $X=S\times C$, where $S$ is a smooth projective surface of general type with nontrivial $\Aut_\QQ(S)$ and $C$ be a smooth projective curve with $g(C)\geq 2$; then $\Aut_\QQ(X)$, containing $\Aut_\QQ(S)\times \{\id_C\}$, is nontrivial, and $q(X)=q(S)+g(C)$ can be arbitrarily large.
	
	Nevertheless, we have the following question about the boundedness of $\Aut_\QQ(X)$ for smooth projective varieties of general type:
	\begin{qu}\label{qu: main}
		Given a positive integer $n$, does there exist a constant $M(n)$, depending on $n$, such that $|\mathrm{Aut}_{\QQ}(X)|\leq M(n)$ for any $n$-dimensional \emph{smooth} projective variety $X$ of general type?
	\end{qu}

	In this paper we give a confirmative partial answer to Question~\ref{qu: main} in dimension $3$.
	\begin{thm}\label{thm: main}
		There is a  positive constant $M$ such that, for any smooth projective threefold of general type satisfying one of the following conditions:
		\begin{enumerate}
			\item[(i)] $q(X)\geq 3$.
			\item[(ii)] $X$ has a Gorenstein minimal model.
		\end{enumerate}
		the inequality $|\Aut_\QQ(X)|\leq M$ holds.
	\end{thm}
	Curiously, the smoothness condition in Theorem~\ref{thm: main} cannot be weakened to, say, the condition of having terminal singularities: There exist projective threefolds $X$ with terminal singularities and maximal Albanese dimension such that $|\Aut_\QQ(X)|$ can be arbitrarily large (\cite[Example~6.3]{Z21}). We observe that the numerically trivial automorphism $\sigma$ appearing there induces a nontrivial translation of the Albanese variety $A_X$. This cannot happen if $X$ were smooth; see Lemma~\ref{prop: fix A}.

	For smooth threefolds with maximal Albanese dimension, we have a more precise statement.
	\begin{thm}[= Corollary~\ref{cor: mAd}]\label{thm: mAd1}
		If $X$ is smooth projective threefold of general type with maximal Albanese dimension, then $|\Aut_\QQ(X)|\leq 4$, and in the equality case we have $\Aut_\QQ(X)\cong(\ZZ/2\ZZ)^2$ and $q(X)=3$.
	\end{thm}
	This generalizes and refines the effective boundedness results of \cite{Z21}. Thanks to \cite[Example~6.1]{Z21}, the bound of Theorem~\ref{thm: mAd1} is optimal.

	By the Hodge decomposition of $H^*(X, \CC)$, $\Aut_\QQ(X)$ acts trivially on $H^i(X, \mo_X)$ and $H^{i}(X, \omega_X)$ for each $i$. It is this condition we are mainly using in the arguments, and it already yields the boundedness we are searching for. Theorems~\ref{thm: main} and \ref{thm: mAd1} are consequences of the following more technical theorem; see also Corollaries~\ref{cor: mAd}, \ref{cor: Ad1}, \ref{cor: Ad2}. Here we allow mild singularities.
	\begin{thm}\label{thm: O trivial}
		There is a constant $M>0$ satisfying the following. Let $X$ be a projective threefold of general type with canonical singularities. Let $\Aut_\mo(X)<\Aut(X)$ be the group of automophisms acting trivially on $H^i(X, \mo_X)$ (or, equivalently, $H^i(X, \omega_X)$) for $i\geq 0$, and let $\Aut_A(X)<\Aut(X)$ the group of automorphisms preserving each fiber of the Albanese map $a_X\colon X\rightarrow A_X$. If $X$ satisfies one of the following conditions:
		\begin{enumerate}
			\item[(i)] $X$ has maximal Albanese dimension, that is, $\dim a_X(X) =3$;
			\item[(ii)] $X$ has Albanese dimension two and either $q(X)\geq 3$ or $A_X$ is a simple abelian surface;
			\item[(iii)]$X$ has Albanese dimension one and $q(X)\geq 2$;
			\item[(iv)] $X$ has a Gorenstein minimal model;
		\end{enumerate}
		then $|\Aut_\mo(X)\cap \Aut_A(X)|\leq M$. In case (i), we have $|\Aut_\mo(X)\cap \Aut_A(X)|\leq 4$, and if the equality holds then $\Aut_\QQ(X)\cong(\ZZ/2\ZZ)^2$ and $q(X)=3$.
	\end{thm}
	
	The third author \cite{Z21} proved the inequality $|\Aut_\mo(X)|\leq 6$ for threefolds of general type satisfying the conditions (i) and (iv) simultaneously; the proof was a straightforward application of the Noether--Severi type inequality and the Miyaoka--Yau inequality bounding the volume of $K_X$ in terms of $\chi(X, \omega_X)$ from below and from above respectively.
	
	In dealing with the case (iv) of Theorem~\ref{thm: O trivial}, we take a similar approach, but the required Noether type inequality in this more general setting is not readily available. For that reason, we prove in the appendix of the paper the existence of a Noether type inequality for log canonical pairs with coefficients from a given subset $\mc\subset(0,1]$ such that $\mc\cup\{1\}$ attains the minimum. This result is of independent interest.
	
	In the case (i) of Theorem~\ref{thm: O trivial}, we use the eventual paracanonical maps \cite{J16, BarPS19} to obtain the effective bound $|\Aut_\mo(X)\cap \Aut_A(X)|\leq 4$. The point is that the eventual paracanonical map $\varphi_X\colon X\dashrightarrow Z $ factors through the quotient map $\pi\colon X\rightarrow X/G$, where $G=\Aut_\mo(X)\cap \Aut_A(X)$. It follows that $|G|\leq \deg \varphi_X$. We can then draw on the explicit description of the cases with $\deg\varphi_X\geq 2$ by \cite{J16}.
	
	In general, there is a factorization of the Albanese map $a_X\colon X\xrightarrow{\pi} Y \xrightarrow{a_Y} A_X$, where $Y=X/G$ with $G=\Aut_\mo(X)\cap \Aut_A(X)$ and $\pi\colon X\rightarrow Y$ is the quotient map. If $\dim a_X(X) <3\leq q(X)$, we use the Chen--Jiang decomposition of $a_{X*}\omega_X$ and $a_{Y*}\omega_Y$ to bound the difference between the ranks of $a_{X*}\omega_X$ and $a_{Y*}\omega_Y$, which in turn gives the required bound on $|G|$. The case with $q(X)=2$ and $A_X$ a simple abelian surface requires a more involved application of the Chen--Jiang decomposition, but the idea is similar.

	\section{Preliminaries}\label{sec: prelim}
	We work over the complex numbers throughout the paper.
	
	In this section we recall the notion of Albanese morphisms for normal projective varieties $X$ as well as the Chen--Jiang decomposition on abelian varieties, and prove some basic properties of numerically trivial automorphisms in relation to their induced actions on the cohomology groups $H^*(X, \mo_X)$, $H^*(X, \omega_X)$ as well as the Albanese variety.
	
	\subsection{The Albanese morphism}
	Let $X$ be a normal projective variety. Then there exists a morphism $a_X: X\rightarrow A_X$ to an abelian variety $A_X$ such that any morphism from $X$ to an abelian variety factors through $a_X$ (\cite[Th\'eor\`eme 5]{Ser59}). This morphism is unique up to translations on $A_X$. We call $A_X$ the Albanese variety of $X$ and $a_X: X\rightarrow A_X$ the Albanese morphism.
	
	When $X$ is a smooth projective variety, it is well known that $\dim A_X=q(X):=h^1(X, \mathcal O_X)$ and  $A_X$ is dual to the Picard variety $\Pic^0(X)$.  In general, if $X$ is singular, then the Albanese morphism $a_{\tilde X}$ of $\tilde X$ may not factor through $\rho$, where $\rho: \tilde X\rightarrow X$ is a desingularization.  But when $X$ has rational singularities, then $a_{\tilde X}$ factors through $\rho$ and $A_X\cong A_{\tilde X}$ (see for instance \cite[Lemma 8.1]{Kaw85}).
	
	Let $X$ be a normal projective variety, and $a_X: X\rightarrow A_X$ the Albanese map with respect to a base point. We call the number $\dim a_X(X)$ the \emph{Albanese dimension} of $X$; if $\dim a_X(X) =X$ then $X$ is said to be \emph{of maximal Albanese dimension}. Let $X\rightarrow S \xrightarrow{a_S}A_X$ be the Stein factorization of $a_X$. By the universal property of the Albanese morphisms, it is easy to check that $A_S=A_X$, and the morphism $a_S\colon S\rightarrow A_X$ is indeed the Albanese morphism of $S$.
	For any $\sigma\in \Aut(X)$, by the universal properties of the Albanese morphisms and of the Stein factorization, there are induced automorphisms $\sigma_S\in \Aut(S)$ and $\sigma_A\in \Aut(A_X)$ such that the following diagram is commutative
	\[
	\begin{tikzcd}
	X \arrow[r, "f"] \arrow[d, "\sigma"] & S \arrow[r, "a_S"] \arrow[d, "\sigma_S"] &  A_X \arrow[d, "\sigma_A"] \\
	X \arrow[r, "f"] & S \arrow[r, "a_S"] &  A_X\\
	\end{tikzcd}
	\]
	In this way, we obtain group homomorphisms
	\begin{equation}\label{eq: induce}
	\psi_S\colon \Aut(X)\rightarrow \Aut(S), \, \sigma\mapsto \sigma_S, \text{\hspace{.5cm}and \hspace{.5cm}  }
	\psi_A\colon\Aut(X)\rightarrow \Aut(A_X),\, \sigma\mapsto \sigma_{A}.
	\end{equation}
	In other words, $ \Aut(X)$ acts on $S$ and $A_X$ such that the morphisms $f$ and $a_S$ are $\Aut(X)$-equivariant.
	
	\subsection{The Chen--Jiang decomposition}
	An important ingredient of the proofs of the main results on irregular threefolds is the Chen--Jiang decomposition theorem. We briefly recall this result.
	
	Let $A$ be an abelian variety. For a coherent sheaf $\cF$ on $A$, we define the  $i$-th \emph{cohomological support loci} for integer $i\geq 0$ to be the following closed subset of $\Pic^0(A)$ with the reduced scheme structure:
	$$V^i(\cF):=\{P\in\Pic^0(A)\mid H^i(A, \cF\otimes P)\neq 0\}.$$ We say that $\cF$ is \emph{GV} if $\codim_{\Pic^0(A)}V^i(\cF)\geq i$ for each $i\geq 1$ and we say that $\cF$ is \emph{M-regular} if $\codim_{\Pic^0(A)}V^i(\cF)> i$ for each $i\geq 1$. These two definitions were introduced respectively by Hacon \cite{Hac04} and by Pareschi--Popa \cite{PP03}. GV sheaves have certain positivity properties.   Hacon proved in \cite{Hac04} that if $0\neq \cF$ is GV, then $V^0(\cF)\neq \emptyset$, that is, there exists $P\in \Pic^0(A)$ such that $H^0(A, \cF\otimes P)\neq 0$. M-regular sheaves have stronger positivity properties. Pareschi and Popa proved in \cite{PP03} that if $\cF$ is M-regular, then $V^0(\cF)=\Pic^0(A)$ and, roughly speaking, $\cF$ can be generated by the twisted global sections.
	
	Given a morphism $f: X\rightarrow A$ from a smooth projective variety to an abelian variety, the main result of \cite{Hac04} states that $R^if_*\omega_X$ is a GV-sheaf for each $i\geq 0$.  The following improvement of Hacon's theorem, which was first proved  when $f$ is generically finite in \cite{CJ18} and  later proved for $f$ general  in \cite{PPS17}, is now called the Chen--Jiang decomposition theorem.
	\begin{thm}\label{thm: decomp}
		Let $f: X\rightarrow A$ be a morphism from a smooth projective variety to an abelian variety.  For each integer $k\geq 0$, there exist finitely many quotient maps $p_I: A\rightarrow B_I$ between abelian varieties with connected fibers, and, for each index $I$, an M-regular sheaf $\cF_I$ on $B_I$ and a torsion line bundle $Q_I\in \Pic^0(A)$ such that $$R^kf_*\omega_X\cong \bigoplus_{I}p_I^*\cF_{I}\otimes Q_I.$$
	\end{thm}
	\begin{rema}\label{rmk:torsion}
		We  remark that under the assumption of Theorem \ref{thm: decomp}, for any torsion line bundle $Q\in \Pic^0(X)$, $R^kf_*(\omega_X\otimes Q)$ also has the Chen--Jiang decomposition (see for instance \cite[Note on page 2449]{PPS17}). Moreover, any direct summand of $R^kf_*(\omega_X)$ has the Chen--Jiang decomposition (see \cite[Proposition 3.6]{LPS20}).
		Similar decompositions also hold for varieties with klt singularities, see for  instance \cite{J20} and \cite{M20}.
	\end{rema}
	\subsection{Some basic properties of numerically trivial automorphisms}
	Let $X$ be a normal projective variety, and $a_X\colon X\rightarrow A_X$ the Albanese morphism. We are interested in various normal subgroups of the automorphism group $\Aut(X)$ of $X$, defined as follows:
	\begin{itemize}
		\item the group of \emph{numerically trivial automorphisms}
		\[\Aut_\QQ(X):=\{\sigma\in\Aut(X) \mid \sigma|_{H^*(X,\QQ)} = \id_{H^*(X,\QQ)}\},\]
		\item the group of \emph{$\mo$-cohomologically trivial automorphisms}
		\[\mathrm{Aut}_{\cO}(X) := \left\{\sigma\in\mathrm{Aut}(X) \mid \sigma|_{H^{*}(X, \cO_X)}=\id_{H^{*}(X, \cO_X)} \right\},\]
		\item the group of \emph{$A$-trivial automorphisms}
		\[\Aut_A(X):=\{\sigma\in\Aut(X) \mid \sigma|_{A_X}  = \id_{A_X}  \},\]
	\end{itemize}
	where $\sigma|_{H^*(X,\QQ)}$, $\sigma|_{H^{*}(X, \cO_X)}$ and  $\sigma|_{A_X}$  denote the induced actions of $\sigma$ on $H^*(X,\QQ):=\bigoplus_i H^i(X,\QQ)$,  $H^{*}(X, \cO_X):=\bigoplus_i H^i(X, \cO_X)$ and $A_X$ respectively.

	If $X$ is  Cohen--Macaulay, then by Serre duality, $\mathrm{Aut}_{\cO}(X)$ acts trivially on $H^i(X, \omega_X)$ for each $i\geq 0$. This is the case when $X$ has rational singularities. The following Lemma~\ref{lem: Hodge} implies that $\mathrm{Aut}_{\mathbb Q}(X)<\mathrm{Aut}_{\cO}(X)$, provided that there is an effective $\RR$-divisor $\Delta$ such that $(X,\Delta)$ is a log canonical pair.  We will see in Proposition~\ref{prop: fix A} that $\mathrm{Aut}_{\mathbb Q}(X)< \mathrm{Aut}_A(X)$ if $X$ is a smooth projective threefold of general type.
	
	\begin{lem}\label{lem: Hodge}
		Let $X$ be a normal projective variety.
		\begin{enumerate}
			\item[(i)] If there is an effective $\RR$-divisor $\Delta$ on $X$ such that $(X, \Delta)$ is log canonical, then $\Aut_\QQ(X)$ acts trivially on $H^i(X, \mo_X)$ for all $i\geq 0$.
			\item[(ii)] If $X$ is smooth, then $\Aut_\QQ(X)$ acts trivially on $H^i(X, \Omega_X^j)$ for all $0\leq i, j\leq \dim X$.
		\end{enumerate}
	\end{lem}
	\begin{proof}
		Note that $\Aut_\QQ(X)$ acts trivially on $H^*(X, \CC) = H^*(X, \QQ)\otimes_\QQ \CC$. If $(X, \Delta)$ is log canonical then the homomorphism $H^i(X, \CC)\rightarrow H^i(X, \mo_X)$ induced by the inclusion of sheaves $\CC\subset \mo_X$ is surjective (\cite{KK10}). It follows that $\Aut_\QQ(X)$ acts trivially on $H^i(X, \mo_X)$. If $X$ is smooth, $\Aut_\QQ(X)$ acts trivially on its direct summands $H^i(X, \Omega_X^j)$ of $H^{i+j}(X, \CC)$ in the Hodge decomposition.
	\end{proof}

	\begin{lem}\label{lem: chi 0}
		Let $X$ be a normal projective variety with rational singularities. Suppose that $\sigma\in\Aut(X)$ induces a trivial action on $H^1(X,\cO_X)$. Then the following holds.
		\begin{enumerate}
			\item[(i)] The induced automorphism $\sigma_A\in \Aut(A_X)$ is a translation.
			\item[(ii)] If  $\sigma$ is of finite order and $\sigma_A\neq \id_{A_X}$, then $\chi(X, \omega_X)=(-1)^{\dim X}\chi(X, \mathcal O_X)=0$ and $e(X)=0$.
		\end{enumerate}
	\end{lem}
	\begin{proof}
		(i) Suppose that $\sigma\in \Aut(X)$ induces the trivial action on $H^1(X, \mo_X)$. Since $a_X^*\colon H^1(A, \mo_{A})\rightarrow H^1(X,\mo_X)$ is an isomorphism satisfying $$\sigma^*\circ a_X^* =  a_X^* \circ\sigma_A^*\colon H^1(A_X, \mo_{A_X}) \rightarrow H^1(X, \mo_X)$$ and $\sigma^*$ is the identity map of $H^1(X,\CC)$, we infer that $\sigma_A^*$ is the identity map of $H^1(A_X, \mo_{A_X})$. It follows that $\sigma_A$ is a translation of $A_X$.
		
		(ii) If $\sigma_A\neq \id_{A_X}$, then it is a nontrivial translation by (i) and hence has no fixed points. It follows that $\sigma$ does not fix any point either. We then conclude by the topological Lefschetz fixed point theorem that $e(X) =0$.
		
		For the vanishing of $\chi(X, \mathcal O_X)$, we take a $\sigma$-equivariant resolution $\rho\colon \tilde X\rightarrow X$ (\cite{EV00}). Since $X$ has rational singularities, we have $a_X\circ \rho = a_X$ and $\chi(X, \mo_X) =\chi(\tilde X, \mo_{\tilde X})$. Let $\tilde\sigma\in \Aut(\tilde X)$ be the lift of $\sigma$. Then $\tilde \sigma$ has no fixed points either.   By the holomorphic Lefschetz fixed point theorem, applied to $\tilde\sigma$, we obtain $\chi(X, \mo_X) =\chi(\tilde X, \mo_{\tilde X})  = 0.$
	\end{proof}
	
	We refer to \cite{KM98} for the notions appearing in Lemma~\ref{lem: descend}.
	\begin{lem}\label{lem: descend}
		Let $(X, \Delta)$ be a projective log canonical pair, and $\varphi_R\colon X\rightarrow Y$ the contraction with respect to a $(K_X+\Delta)$-negative extremal ray $R$ of the Kleiman--Mori cone of $(X, \Delta)$. Then any numerically trivial automorphism $\sigma\in\Aut_\QQ(X)$ induces an automorphism $\sigma_Y\in \Aut(Y)$. Moreover, the following holds.
		\begin{enumerate}
			\item[(i)] If $Y$ has rational singularities, then $\sigma_Y$ is $\mo$-cohomologically trivial, that is, $\sigma_Y\in \Aut_\mo(Y)$.
			\item[(ii)] If $X$ and $Y$ are smooth, then $\sigma_Y$ is numerically trivial, that is, $\sigma_Y\in \Aut_\QQ(Y)$.
		\end{enumerate}
	\end{lem}
	\begin{proof}
		By the contraction theorem, there is a line bundle $L$ on $X$ such that $L=\varphi_R^*L_Y$ for some ample line bundle $L_Y$ on $Y$. Since $\sigma$ acts trivially on $H^2(X,\QQ)$, it preserves the cohomology class of $L$ in $H^2(X,\QQ)$. Note that homological euqivalence coincides with algebraic equivalence for Cartier divisors, we infer that $\sigma^*L$ is numerically equivalent to $L$. For any curve $C$ whose numerical class generates the extremal ray $R$, using the projection formula, 
		\[
		L_Y\cdot(\varphi_R\circ\sigma)_*(C)=L\cdot\sigma_*(C)=(\sigma^*L)\cdot C=0,
		\]
		we get that $C$ is contracted by $\varphi_R\circ\sigma$. By the rigidity of birational contractions, there exists a morphism $\sigma_Y\colon Y\to Y$ such that the following diagram is commutative
		\[
		\begin{tikzcd}
		X \arrow[r, "\sigma"] \arrow[d, "\varphi_R"']& X\arrow[d, "\varphi_R"] \\
		Y\arrow[r, "\sigma_Y"]&Y
		\end{tikzcd}
		\]
		Since $\sigma$ is an isomorphism, we conclude that $\sigma_Y$ is an isomorphism of $Y$.
		
		(i) If $Y$ has rational singularities, then we have isomorphisms 
		\[
		H^i(X, \mo_X) \cong H^i(Y, \mo_Y)
		\]
		for each $i\geq 0$ which is compatible with induced actions of $\sigma$ and $\sigma_Y$ on $H^i(X, \mo_X)$ and $H^i(Y, \mo_Y)$ respectively. Since $\sigma\in \Aut_\mo(X)$, we infer that $\sigma_Y\in \Aut_\mo(Y)$.
		
		(ii) By the decomposition theorem for resolutions (\cite{dCM09}), $\varphi_R^*\colon H^*(Y, \CC)\rightarrow H^*(X, \CC)$ is injective. Thus the numerical triviality of $\sigma$ implies that of $\sigma_Y$.
	\end{proof}
	
	\begin{prop}\label{prop: fix A}
		Let $X$ be a smooth projective threefold of general type. Then $\Aut_\QQ(X)$ acts trivially on $A_X$.
	\end{prop}
	\begin{proof}
		The statement is trivial when $q(X)=0$. We can thus assume that $q(X)>0$.
		
		Take $\sigma\in\Aut_{\QQ}(X)$ and let $\sigma_A\in\Aut_{\QQ}(A_X)$ be the induced automorphism of $A_X$. Then $\sigma_A$ is a translation of $A_X$ by Lemmas~\ref{lem: Hodge} and \ref{lem: chi 0}. If $X$ has a smooth minimal model, then the Miyaoka--Yau inequality, $\vol(K_X)\leq 72\chi(\mo_X)$ implies that $\chi(\mo_X)>0$. By Lemma~\ref{lem: chi 0}, $\sigma_A$ is trivial.
		
		Now we assume that the minimal models of $X$ are singular. In order to show that $\sigma_A=\id_{A_X}$, it suffices to show that $\sigma_A$ fixes a point.  Consider a minimal model program
		\[
		\begin{tikzcd}
		X=X_0 \arrow[dashed]{r}{\rho_0}  & X_1  \arrow[dashed]{r}{\rho_1} &X_2 \arrow[dashed]{r}{\rho_2} &\cdots \arrow[dashed]{r}{\rho_{n-1}} &X_n
		\end{tikzcd}
		\]
		where for $0\leq i\leq n$, $X_i$ has terminal singularities, the $\rho_i$ are either divisorial extremal contractions or flips, and $K_{X_n}$ is nef. As we have explained in last section,  the induced maps $a_i\colon X_i\dashrightarrow A_X$ are  indeed the corresponding Albanese morphisms  of $X_i$ and the above minimal model program is over $A_X$.
		
		By assumption, $X_0=X$ is smooth and $X_n$ is singular, so we there is an index $i_0$ such that $X_{i_0}$ is smooth and $X_{i_0+1}$ is singular. By the classification of extremal contractions in dimension three (\cite{Mo82}; see also \cite[Theorem1.32]{KM98}), the birational maps $\rho_i\colon X_i \dashrightarrow X_{i+1}$ are divisorial contractions for $i\leq i_0$ and, moreover,  $\rho_{i_0}\colon X_{i_0}\rightarrow X_{i_0+1}$ is a morphism contracting a divisor $E\subset X_{i_0}$ to a point $p\in X_{i_0+1}$. By Lemma~\ref{lem: descend}, $\sigma$ descends to a numerically trivial automorphism $\sigma_{i_0}\in\Aut_{\QQ}(X_{i_0})$. In particular, $\sigma_{i_0}$ preserves the exceptional locus $\Exc(\rho_{i_0})$, and this implies that the induced automorphism $\sigma_A$ fixes the point $a_{i_0}(E) =a_{i_0+1}(p)\in A_X$.
	\end{proof}

	\section{Irregular threefolds of general type}
	
	In this section, we prove Theorem~\ref{thm: O trivial}, cases (i)-(iii).  Let $X$ be an irregular projective threefold of general type with canonical singularities and $a_X\colon X\rightarrow A_X$ its Albanese map with respect to a base point.

	We proceed according to the Albanese dimension $\dim a_X(X)$.
	
	\subsection{The case $\dim a_X(X)=3$}
	In this subsection we assume that $\dim a_X(X)=3$, so $X$ is of maximal Albanese dimension.
	\begin{thm}\label{thm: mAd}
		Let $X$ be a projective threefold of general type with canonical singularities and of maximal Albanese dimension. Then  $|\mathrm{Aut}_{\cO}(X)\cap \mathrm{Aut}_A(X)|\leq 4$, and if the equality holds then $\mathrm{Aut}_{\cO}(X)\cap \mathrm{Aut}_A(X)\cong (\ZZ/2\ZZ)^2$ and $q(X)=3$.
	\end{thm}
	\begin{proof}
		Let $G:=\mathrm{Aut}_{\cO}(X)\cap \mathrm{Aut}_A(X)$.
		Then  the Albanese morphism $a_X$ factors through the quotient morphism $\pi: X\rightarrow Y:=X/G$.
		
		If $\chi(X, \omega_X)=0$, then $a_X$ is birationally an $(\mathbb Z/2\ZZ)^2$-cover by \cite{CDJ14}. Thus $G$ is an (abelian) subgroup of $(\mathbb Z/2\ZZ)^2$.
		
		If $\chi(X, \omega_X)>0$, then $\chi(X, \omega_X)=\chi(Y, \omega_Y)>0$.  By \cite{J16}, the eventual paracanonical map $\varphi_X\colon X\dashrightarrow Z$ factors through the quotient map as $\varphi_X\colon X\xrightarrow{\pi}Y\dashrightarrow  Z$, where $Y\dashrightarrow  Z$ is the eventual paracanonical map $\varphi_Y$ of  $Y$ (\cite[Proposition~1.5]{J16}). Note that, $\deg\varphi_X\leq 8$ by \cite[Theorem 1.7]{J16}, and hence
		\begin{equation}\label{eq: deg event paracan}
		|G|=\deg\pi  = \deg\varphi_X / \deg\varphi_Y \leq \frac{8}{\deg\varphi_Y}.
		\end{equation}
		If $\deg\varphi_Y\geq 2$, then either $|G| \leq 3$ or $|G|=4$ and $\deg \varphi_X=8$. In the latter case, $\varphi_X$ is birationally a $(\ZZ/2\ZZ)^3$-cover by \cite[Theorem 1.7]{J16}. If $\deg\varphi_Y=1$, then $\varphi_Y$ is birational and hence
		\[
		\chi(Z, \omega_Z) =\chi(Y, \omega_Y) = \chi(X, \omega_X)>0.
		\]
		By \cite[Theorem 1.7]{J16} again, $\varphi_X$  induces a Galois field extenstion $\mathbb C(X)/\mathbb C(Z)$ with Galois group $\ZZ/2\ZZ$ or $\ZZ/3\ZZ$, or $(\ZZ/2\ZZ)^2$. In conclusion, we have $|G|\leq 4$, and if the equality holds then $G\cong (\ZZ/2\ZZ)^2$.
		
		In the case $G\cong (\ZZ/2\ZZ)^2$, we will show that $q(X)= 3$. When $\chi(X, \omega_X)=0$, we automatically have $q(X)=3$ by \cite{CDJ14}. We can thus assume that $\chi(X, \omega_X)=\chi(Y, \omega_Y)>0$.
		If $q(X)>3$, we are in the situation of \cite[Subsection 4.1, Case 1]{J16}. By \cite[Proposition 4.4]{J16}, the quotient map $\pi$ is birationally equivalent to the eventual paracanonical map of $X$. Let $V:=a_X(X)$ be the Albanese image of $X$ with the reduced scheme structure. Since $\dim A_X>3$, $V$ is fibred by an abelian subvariety $K$ such that the quotient $V/K$ is not fibred by any positive dimensional abelian subvariety of $A_X/K$ and any desingularization of $V/K$ is of general type (see \cite[Theorem 10.9]{Ue75}). Let $B:=A/K$ and $V_B:=V/K$.
		We have the following picture:
		\begin{eqnarray*}
			\xymatrix{
				X\ar[drr]_{f}\ar[r]^{\pi} & Y\ar[dr]^{h}\ar[r]& V\ar[d]\ar@{^{(}->}[r]&A_X\ar[d]\\
				& &V_B \ar@{^{(}->}[r] & B.}
		\end{eqnarray*}
		Let $t\in V_B$ be a general point and let $X_t$ and $Y_t$ be   the fibers of $f$ and $h$ over $t$ respectively. We remark that $X_t$ and $Y_t$ may be reducible and each component of $X_t$ or $Y_t$ is of general type. Fix a general torsion line bundle $Q\in \Pic^0(A_X)$. Note that $A_X$ is the Albanese variety of both $X$ and $Y$. We can simply regard $Q$ as a general torsion line bundle on $X$ or on $Y$. By generic vanishing,
		\[
		h^0(X, \omega_X\otimes Q)=\chi(X, \omega_X\otimes Q)=\chi(X, \omega_X)=\chi(Y, \omega_Y)=\chi(Y, \omega_Y\otimes Q)=h^0(Y, \omega_Y\otimes Q).
		\]
		Moreover, since $Q$ is general, we also have $h^i(X_t, \omega_{X_t}\otimes Q)=h^i(Y_t, \omega_{Y_t}\otimes Q)=0$. Thus both $R^if_*(\omega_X\otimes Q)$ and $R^ih_*(\omega_Y\otimes Q)$ are supported on a proper subset of $V_B$. However, by the main theorem of \cite{Ko86a}, these sheaves are torsion-free on $V_B$ if they are not zero. Thus,
		$R^if_*(\omega_X\otimes Q)=R^ih_*(\omega_Y\otimes Q)=0$ for $i\geq 1$. Therefore, for $Q'\in \Pic^0(B)$ general,
		\begin{multline}\label{equal}
		h^0(V_B, f_*(\omega_X\otimes Q)\otimes Q')=\chi(X, \omega_X\otimes Q\otimes Q') \\
		=\chi(Y, \omega_Y\otimes Q\otimes Q')=h^0(V_B, h_*(\omega_Y\otimes Q)\otimes Q').
		\end{multline}
		
		We claim that the inclusion of coherent sheaves $h_*(\omega_Y\otimes Q)\subset f_*(\omega_X\otimes Q)$ is an isomorphism. Assume the contrary, and let $\cQ$ be the \emph{nonzero} quotient sheaf. Then $V^0(\cQ)$ is a proper subset of $\Pic^0(B)$ by \eqref{equal}. On the other hand, by \cite{Ko86a}, these two sheaves are torsion-free and since $V_B$ is not fibred by positive dimensional abelian subvariety, using the Chen--Jiang decomposition for $h_*(\omega_Y\otimes Q)$ and $f_*(\omega_X\otimes Q)$ (see \ref{thm: decomp} and Remark \ref{rmk:torsion}), both of them are M-regular. It follows that $\cQ$ is also M-regular and hence $V^0(\cQ) = \Pic^0(B)$, which is a contradiction.
		
		The equality $h_*(\omega_Y\otimes Q)= f_*(\omega_X\otimes Q)$ implies that $\chi(X_t, \omega_{X_t})=\chi(Y_t, \omega_{Y_t})$.
		Note that $\pi_t\colon X_t\rightarrow Y_t$ is again of degree $4$. The case $\dim X_t=1$ is impossible by Riemann--Roch. The remaining case $\dim X_t=2$ is impossible  by \cite[Theorem 1.6]{J16} or \cite[Proposition 3.3]{BarPS19}.
	\end{proof}
	
	\begin{cor}\label{cor: mAd}
		Let $X$ be any smooth projective threefold of general type  with maximal Albanese dimension. Then $|\Aut_\QQ(X)|\leq 4$, and if the equality holds then $ \Aut_\QQ(X)\cong (\ZZ/2\ZZ)^2$ and $q(X)=3$.
	\end{cor}
	\begin{proof}
		By Lemmas~\ref{lem: Hodge} and \ref{prop: fix A}, we know that $\Aut_\QQ(X)$ is contained in the group $\Aut_{\cO}(X)\cap \Aut_A(X)$ of Theorem~\ref{thm: mAd}.
	\end{proof}
	Note that the case $\mathrm{Aut}_\QQ(X)\cong (\ZZ/2\ZZ)^2$ is realized by an unbounded series of threefolds by \cite[Example 6.1]{Z21}.

	\subsection{The case $\dim a_X(X)<3$}
	In this subsection we deal with irregular threefolds of general type which are not of maximal Albanese dimension. We will denote by $G:=\Aut_\mo(X)$ the $\mo$-cohomologically trivial automorphism subgroup, and $\pi\colon X\rightarrow Y:=X/G$ the quotient morphism.
	\begin{lem}\label{lem: Ad1}
		There is a positive constant $M_1$ such that the following holds. Let $f\colon X\rightarrow B$ be a fibration from a smooth projective threefold of general type onto a smooth projective curve $B$ with $g(B)\geq 2$. Then $|\Aut_{\cO}(X)|\leq M_1$.
	\end{lem}
	\begin{proof}
		Since $f^*\colon H^1(B, \mo_B)\rightarrow H^1(X, \mo_X)$ is injective and
		$G:=\Aut_\mo(X)$ acts trivially on $H^1(X, \mo_X)$, $G$ induces a trivial action on $H^1(B, \mo_B)$. Since $g(B)\geq 2$, $G$ acts trivially on $B$, that is, $G\subset \Aut_B(X)$, the group of automorphisms preserving each fiber of $f$.
		
		Let $F$ be a general fiber of $f$. Then $G$ acts on $F$ via the injective homomorphism $G\hookrightarrow \Aut(F)$, $\sigma\mapsto \sigma|_F$. Let $\pi\colon X\rightarrow Y:=X/G$ and
		$\pi_F: F\rightarrow F' := F/G$ be the quotient morphisms. Since $B$ is of general type, by Theorem \ref{thm: decomp} or  \cite[Lemma 2.1]{JLT14}, both $a_{X*}\omega_X$ and $a_{Y*}\omega_Y$ are M-regular. Moreover, $p_g(X)=p_g(Y)$, thus $a_{Y*}\omega_Y=a_{X*}\omega_X$. Hence  $p_g(F)=p_g(F')$.
		
		By \cite[Propositions 2.1 and 4.1]{Bea79},  if $\chi(F, \mathcal O_F)\geq 31$, then the canonical map $\varphi_F$ of $F$ is either of degree $\leq 9$ or is a fibration such that the genus $g$ of a general fiber is $\leq 5$. Thus in this case we obtain
		\[
		|G|\leq
		\begin{cases}
		\deg\varphi_F \leq 9 & \text{if $\varphi_F$ is generically finite}\\
		84(g-1) \leq 336 &\text{if $\varphi_F$ is composed with a pencil}
		\end{cases}
		\]
		where the inequality in second case follows from the Hurwitz bound of the full automorphism group of a curve with genus $g\geq 2$. On the other hand, surfaces $F$ with $\chi(F, \mathcal O_F)<31$ form a birationally bounded family and thus the order of their automorphism groups are bounded by a constant. As a conclusion, there exists a constant $M_1$ such that $|G|\leq M_1$.
	\end{proof}
	
	\begin{cor}\label{cor: Ad1}
		The inequality $|\Aut_{\cO}(X)|\leq M_1$ holds for any smooth projective threefold $X$ of general type with $q(X)\geq 2$ and $\dim a_X(X)=1$.
	\end{cor}
	\begin{proof}
		Since $q(X)\geq 2$ and $\dim a_X(X)=1$, we infer that the Albanese image of $X$ factors as $a_X\colon X\xrightarrow{f} C \hookrightarrow A_X$, where $g(C)=q(X)$ and $f$ is a fibration. Now apply Lemma~\ref{lem: Ad1}.
	\end{proof}
	
	\begin{rema}We can indeed work out an explicit bound for $M_1$. Note that in the proof of Lemma~\ref{lem: Ad1}, when $\chi(F, \mathcal O_F)<31$, we have $\mathrm{vol}(K_F)\leq 9\chi(F, \mathcal O_F)\leq 180$ by the Bogomolov--Miyaoka--Yau inequality. Thus $|G|\leq 42^2\mathrm{vol}(K_F)\leq 317520$ by \cite{X95}, and we may take $M_1=317520$.
	\end{rema}
	
	\begin{prop}\label{prop: Ad2}
		Let $X$ be a smooth threefold of general type with $q(X)\geq 3$ and $\dim a_X(X)=2$. Then $|\Aut_{\cO}(X)|\leq M_1$.
	\end{prop}
	\begin{proof}
		Since $a_X$ is not surjective, the image $Z$ of the Albanese morphism (or rather its desingularization) has Kodaira dimension $\geq 1$.
		
		Let $H$ be the image of $\Aut_{\cO}(X)$ in $\Aut(A_X)$; see \eqref{eq: induce}. By Lemma \ref{lem: chi 0}, $H$ acts on $A_X$ by translations. Thus the quotient morphism $A_X\rightarrow A_X/H$ is \'etale.
		
		Consider the commutative diagram
		\begin{eqnarray*}
			\xymatrix{
				X\ar[r]^{\pi}\ar@/_2pc/[dd]_{a_X}\ar[d] & Y\ar[d]\ar@/^2pc/[dd]^{a_Y}\\
				Z\ar[r]\ar@{^(->}[d]& Z/H\ar@{^(->}[d]\\
				A_X \ar[r] & A_X/H}
		\end{eqnarray*}
		where $Y=X/G$ and $\pi\colon X\rightarrow Y$ is the quotient map. Since the morphism $Z\rightarrow Z/H$ is \'etale, the Kodaira dimension of the desingularization of $Z/H$ is the same as the Kodaira dimension of the desingularization of $Z$.
		
		If $Z$ is of general type, then so is $Z/H$. Hence $a_{Y*}\omega_Y$ and $a_{Y*}\pi_*\omega_X$ are M-regular sheaves supported on $Z/H$ by Theorem~\ref{thm: decomp}. By the trace map, we write $\pi_*\omega_X=\omega_Y\oplus \cM$. Suppose that $G$ is non-trivial. Then, since the general fibers of $a_Y$ are curves, the rank of $a_{Y*}\pi_*\omega_X$  is strictly larger than that of $a_{Y*}\omega_Y$. Hence $a_{Y*}\cM$ is non-zero and, being M-regular, has a non-zero global section. But then $p_g(X)=h^0(A_X/H, a_{Y*}\pi_*\omega_X)>h^0(A_X/H, a_{Y*}\omega_Y) = p_g(Y)$, which is a contradiction.

		If $Z$ has Kodaira dimension $1$, then it is an elliptic fiber bundle over a curve $C$ with genus $g(C)\geq 2$, by \cite[Theorem 10.9]{Ue75}. Therefore, the Stein factorization $X\to \tilde{C}$ of the natural map $X\to C$ is a fibration with $g(\tilde{C})\geq 2$. It follows that $|G|\leq M_1$, where $M_1$ is as in Lemma~\ref{lem: Ad1}.
	\end{proof}
	
	Using the same argument for Corollary~\ref{cor: mAd}, we obtain
	\begin{cor}\label{cor: Ad2}
		Let $X$ be a  smooth projective threefold of general type such that $\dim a_X(X)=2$ and $q(X)\geq 3$. Then $|\Aut_\QQ(X)|\leq M_1$, where $M_1$ is the constant appearing in Lemma~\ref{lem: Ad1}.
	\end{cor}
	
	The rest of this section is devoted to the proof of the following
	\begin{prop}\label{prop: Ad2 simple}
		There is a positive constant $M_2$ such that the following holds. Let $X$ be a smooth threefold of general type with $q(X)=\dim a_X(X)=2$ such that its Albanese variety $A_X$ is a simple abelian surface. Then $|\Aut_{\cO}(X)|\leq M_2$.
	\end{prop}
	We need some preparations for the proof.
	
	Assume that $X$ is as in Proposition~\ref{prop: Ad2 simple}. Since $X$ is of general type, $V^0(a_{X*}\omega_X)$ generates $\Pic^0(A_X)$ by \cite{JS15}. By the simplicity of  $\Pic^0(A_X)$, we conclude that
	\begin{equation}\label{eq: V0}
	V^0(a_{X*}\omega_X)=\Pic^0(A_X)
	\end{equation}
	In particular, $p_g(X)>0$.
	
	Let $f\colon X\rightarrow S$ be the Stein factorization of $a_X\colon X\rightarrow A_X$. Then $S$ is a normal projective surface which is finite over the simple abelian surface $A_X$.
	Consider
	$G:=\Aut_{\cO}(X)$
	and
	its
	images
	$\psi_S(G)\subset\Aut(S)$
	and
	$\psi_A(G)\subset\Aut(A_X)$.
	Then we have the following commutative diagram
	\begin{equation}\label{eq: quotient}
	\begin{tikzcd}
	X \arrow[r, "f"] \arrow[d,"\pi"'] \arrow[rr, bend left =45, "a_X"]& S \arrow[d, "\pi_S"] \arrow[r, "a_S"] & A_X\arrow[d, "\pi_A"]\\
	Y \arrow[r, "h"]  \arrow[rr, bend right =45, "a_Y"']& T \arrow[r, "a_T"] & B
	\end{tikzcd}
	\end{equation}
	where $Y=X/G,\, T=S/\psi_S(G),\, B=A_X/\psi_A(G)$ are the quotient varieties, and $\pi,\,\pi_S,\, \pi_A$ are the quotient maps.  One can check that the morphism $Y\rightarrow B$ is indeed the Albanese morphism of $Y$.
	
	\begin{lem} \label{lem: B=A}
		$\psi_A\colon G\rightarrow \Aut(A_X)$ is trivial, so $B=A_X$ and the quotient map $\pi_A$ is the identity.
	\end{lem}
	\begin{proof}
		Assume on the contrary that $\psi_A(G)$ is nontrivial. Note that $\psi_A(G)$ consists of translations of $A_X$, so $\pi\colon A_X\rightarrow B=A_X/\psi_A(G)$ is \'etale.
		Let $Y'\cong X/\ker(\psi_A)$. Then $Y=Y'/\bar G$, where $\bar G =G /\ker(\psi_A)$, and we have a commutative diagram:
		\begin{eqnarray*}
			\xymatrix{
				X\ar[r]^{\pi'}\  \ar[d]_{a_X}& Y'\ar[r]^{\pi''}\ar[d]_{a_{Y'}}& Y \ar[d]^{a_Y}\\
				A_X\ar@{=}[r] & A_X\ar[r]^{\pi_A} & B.
			}
		\end{eqnarray*}
		where $\pi'\colon X\rightarrow Y'$ and $\pi''\colon Y'\rightarrow Y$ are the quotient maps. By the universal property of fiber products, one sees easily that $Y'= Y\times_B A_X$, and $\pi'\colon Y'\rightarrow Y$ and $a_{Y'}\colon Y'\rightarrow A_X$ are the projections under this identification. Since $\displaystyle \pi_{A*}\mo_A = \bigoplus_{P\in \ker \pi_A^*} P$, where $\pi_A^*\colon \Pic^0(B)\rightarrow \Pic^0(A_X)$ is the pull-back, we have
		\begin{equation}\label{eq: KY'}
		\pi''_*\omega_{Y'}=\bigoplus_{P\in \ker \pi_A^*} (\omega_Y\otimes P).
		\end{equation}
		
		Since $G$ acts trivially on $H^0(X, \omega_X)$, we have 
		\[
		H^0(X, \omega_X)\cong H^0(Y', \omega_{Y'})\cong H^0(Y, \omega_Y).
		\] 
		Thus this implies that for $P$ non-trivial in \eqref{eq: KY'}, $H^0(Y, \omega_Y\otimes P)=0$. Thus $V^0(a_{Y*}\omega_Y)$ is a proper subset of $\Pic^0(B)$. Since $A_X$ and hence $B$ is simple, so are $\Pic^0(A_X)$ and $\Pic^0(B)$. Thus $V^0(a_{Y*}\omega_Y)$ consists of finitely many points.  By Theorem \ref{thm: decomp},
		\begin{equation}\label{eq: CJ Y}
		a_{Y*}\omega_Y = \bigoplus_i Q_i,
		\end{equation}
		where $Q_i\in \Pic^0(B)$ are torsion line bundles on $B$.
		
		On the other hand, as a consequence of \eqref{eq: V0}, we have $p_g(Y)=p_g(X)>0$. It follows that some $Q_i$ in the decomposition \eqref{eq: CJ Y} must be $\mathcal O_B$. By the Leray spectral sequence, $q(Y)-q(B)$ is the number of trivial line bundles in \eqref{eq: CJ Y}, and hence
		\[
		q(Y) =q(B) + (q(Y) - q(B)) >2
		\]
		which contradicts the assumption that $q(Y)=q(X)=2$.
	\end{proof}
	
	By Lemma~\ref{lem: B=A}, the commutative diagram \eqref{eq: quotient} simplifies to
	\begin{equation}\label{diag: ST}
	\begin{tikzcd}
	X \arrow[r, "f"] \arrow[d,"\pi"'] \arrow[drr, bend left =90, "a_X"]& S \arrow[d, "\pi_S"'] \arrow[dr, "a_S"] & \\
	Y \arrow[r, "h"]  \arrow[rr, bend right =45, "a_Y"']& T \arrow[r, "a_T"] & A_X
	\end{tikzcd}
	\end{equation}
	Note that every variety in \eqref{diag: ST} have $A_X$ as its Albanese variety.
	
	\begin{lem}\label{lem: M}
		Let $\pi_*\omega_X=\omega_Y\oplus \cM$ be the splitting induced by the trace map. Then the following holds.
		\begin{enumerate}
			\item[(i)]$H^i(Y, \cM)=0$ and $H^i(A_X, R^ja_{Y*}\cM)=0$ for all $i,j \geq 0$.
			\item[(ii)] $R^j a_{Y*}\cM =\bigoplus_k Q_{k}^j$, where $Q_{k}^j\in \Pic^0(A_X)$ are nontrivial torsion line bundles on $A_X$ for $j\geq 0$.
		\end{enumerate}
	\end{lem}
	\begin{proof}
		Since $G$ acts trivially on $H^i(X, \omega_X)$, we have $H^i(X, \omega_X)\cong H^i(Y, \omega_Y)$, hence $H^i(Y, \cM)=0$ for all $i\geq 0$. Pushing the decomposition $\pi_*\omega_X=\omega_Y\oplus \cM$ forward to $A_X$ by $a_{Y*}$, we obtain
		\begin{equation}\label{eq: M}
		a_{X*}\omega_X=a_{Y*}\omega_Y\oplus a_{Y*}\cM.
		\end{equation}
		
		Note that $H^0(A, a_{Y*}\cM)=H^0(Y, \cM)=0$. By Theorem \ref{thm: decomp} and Remark \ref{rmk:torsion}, taking the simplicity of $A_X$ into account, we infer that
		\[
		a_{Y*}\cM = \bigoplus_{k} Q_k
		\]
		where $Q_k\in \Pic^0(A)$ are nontrivial torsion line bundles on $A_X$.  This implies that (\cite[Corollary~3.2]{Hac04})
		\[
		H^i(A_X, a_{Y*}\cM)= \bigoplus_{k} H^i(A_X, Q_k) = 0\, \text{ for all } i\geq 0.
		\]
		Now it follows from the Leray spectral sequences for $H^i(Y,\cM)$ that 
		\[
		H^i(A_X, R^1a_{Y*}\cM)=0
		\]
		for all $i\geq 0$. Using the Chen--Jiang decomposition of $R^1a_{Y*}\cM$ and the simplicity of $A_X$, we obtain
		\[
		R^1a_{Y*}\cM =\bigoplus_k Q_{k}^1
		\]
		where the $Q_{k}^1$ are all nontrivial torsion line bundles on $A_X$.
		
		Finally, we remark that  $R^ja_{Y*}\cM\subset R^ja_{Y*}\omega_Y=0$ for $j\geq 2$ by \cite[Theorem~2.1]{Ko86a}.
	\end{proof}
	
	\begin{lem}\label{lem: ST}
		The homomorphism $\psi_S\colon G\rightarrow \Aut(S)$ is trivial, or equivalently, the quotient morphism $\pi_S\colon S\rightarrow T=S/\psi_S(G)$ in \eqref{diag: ST} is an isomorphism.
	\end{lem}
	\begin{proof}
		
		If $\kappa(S)=0$, then $S$ is an abelian surface and $S\rightarrow A_X$ is an isomorphism by the universal properties of the Stein factorization and the Albanese morphism.
		As a consequence, $S=T$.
		
		If $\kappa(S)=1$ then there is a pencil of (possibly singular) elliptic curves on $S$, whose image in $A_X$ is necessarily again a pencil of elliptic curves, contradicting the simplicity of $A_X$.
		
		Suppose now $\kappa(S)=2.$ Since $\pi_S$ is finite, in order to show that $\pi_S$ is an isomorphism, it suffices to show that $\pi_S$ is birational. For this purpose, we may pass to smooth higher birational models $X', Y', S', T'$ of $X, Y, S$ and $T$ respectively, and by abuse of notation, assume that $S$ and $T$ are themselves smooth. We will show that $\chi(S, \omega_{S}) =\chi(T, \omega_{T})$. Grant this for the moment. Then, since the eventual paracanonical map of $S$ factors through $S\rightarrow T$ by \cite[Proposition~1.5]{J16}, we can apply \cite[Theorem~1.6]{J16} to conclude that $\pi_S\colon S\rightarrow T$ is birational.
		
		It remains to show that $\chi(S, \omega_{S}) = \chi(T, \omega_{T})$. By \cite[Proposition~7.6]{Ko86a}, we have $R^1f_*\omega_X=\omega_{S}$ and $R^1h_*\omega_Y=\omega_{T}$. Therefore,
		\begin{equation}\label{eq: XSYT}
		\chi(X, \omega_X)=\chi(S, f_*\omega_X)-\chi(S, \omega_{S}) \text{ and } \chi(Y, \omega_Y)=\chi(T, h_*\omega_Y)-\chi(T, \omega_{T}).
		\end{equation}
		As in Lemma~\ref{lem: M}, we have a splitting $\pi_*\omega_X=\omega_Y\oplus\cM$ such that
		\begin{equation}\label{eq: vanish M}
		H^i(A_X, a_{Y*}\cM) =0 \text{ for any } i \geq 0
		\end{equation}
		Since $a_{S*}f_*\omega_X = a_{X_*}\omega_X=a_{Y*}\omega_Y\oplus a_{Y*}\cM=a_{T*}h_*\omega_Y\oplus a_{Y*}\cM$, we obtain
		\begin{multline}\label{eq: ST}
		\chi(S, f_*\omega_X)=\chi(A_X, a_{S*}f_*\omega_X)=\chi(A_X, a_{T*}h_*\omega_Y) + \chi(A_X, a_{Y*}\cM) \\
		=\chi(A_X, a_{T*}h_*\omega_Y) = \chi(T, h_*\omega_Y)
		\end{multline}
		where the first and the last equalities are due to the vanishing of $R^1a_{S*}(f_*\omega_X)$ and $R^1a_{T*}(h_*\omega_Y)$ by \cite[Theorem~3.4 (iii)]{Ko86b} and the third equality is by \eqref{eq: vanish M}. Moreover, since $h^i(X, \omega_X)=h^i(Y, \omega_Y)$ for each $i\geq 0$, we have $\chi(X, \omega_X)=\chi(Y, \omega_Y)$. Combining this with \eqref{eq: XSYT} and \eqref{eq: ST}, we obtain $\chi(T, \omega_{T}) = \chi(S, \omega_{S})$.
	\end{proof}
	
	By Lemma~\ref{lem: ST}, the commutative diagram~\eqref{diag: ST} further simplifies to
	\begin{equation}\label{diag: S}
	\begin{tikzcd}
	X \arrow[dr, "f"] \arrow[d,"\pi"'] \arrow[drr, bend left =45, "a_X"]&  & \\
	Y \arrow[r, "h"]  \arrow[rr, bend right =45, "a_Y"']& S \arrow[r, "a_S"] & A_X
	\end{tikzcd}
	\end{equation}

	\begin{proof}[Proof of Proposition~\ref{prop: Ad2 simple}]
		In the diagram~\eqref{diag: S}, the general fiber $C$ of $f$ is preserved by the action of $G$. Let $D=\pi(C)$ be the fiber of $h$ over the point $f(C)\in S$. Then $D\cong C/G$ and $\pi|_C \colon C \rightarrow D$ is the quotient map under the action of $G$. By the Riemann--Hurwitz formula, we have
		\begin{equation}\label{eq: RH}
		2g(C)-2 = |G|(2g(D)-2)+\delta
		\end{equation}
		where $\delta$ is the degree of the ramification divisor of $\pi|_C \colon C \rightarrow D$. Also, we have
		\[
		\rk(a_{Y*}\cM) = \rk(a_{S*}h_*\cM) = (\deg a_S)\cdot \rk(h_*\cM) =  (\deg a_S)(g(C)-g(D)),
		\]
		where $\rk(\cdot)$ denotes the rank of a sheaf. One has $p_g(Y)=p_g(X)>0$ by \eqref{eq: V0} and hence $g(D)\geq 1$.
		
		By Lemma~\ref{lem: M}, there is an abelian \'etale cover $\mu: \wA_X\rightarrow A_X$ such that $\mu^*(R^1a_{Y*}\cM)$
		is a direct sum of trivial line bundles.  Consider the base change of $X\rightarrow Y\rightarrow S\rightarrow A_X$ by $\mu\colon\tilde A_X\rightarrow A_X$,
		\[
		\begin{tikzcd}
		\tilde X\arrow[r, "\tilde \pi"] \arrow[d, "\mu_X"] \arrow[rr, bend left = 45, "\tilde f"]& \tilde Y\arrow[r, "\tilde h"]\arrow[d, "\mu_Y"] &\tilde S \arrow[r, "\tilde a_{S}"] \arrow[d, "\mu_S"]& \tilde A_X\arrow[d, "\mu"]\\
		X\arrow[r, "\pi"] \arrow[rr, bend right = 45, "f"]&  Y\arrow[r, "h"]&S \arrow[r, "a_S"] & A_X
		\end{tikzcd}
		\]
		where $\tilde X=X\times_{A_X} \tilde A_X,\, \tilde Y\times_{A_X} \tilde A_X$ and $\tilde S=S\times_{A_X}\tilde A_X$, and the squares are those of fiber products.
		
		By \cite[Corollary 3.2]{Ko86b}, we have
		\begin{eqnarray*}
			H^2(\wX, \omega_{\wX})&\cong & H^1({\tilde S}, R^1{\tilde f_*}\omega_{\wX} )\oplus H^2({\tilde S}, {\tilde f_*}\omega_{\wX})\\
			&\cong& H^1({\tilde S}, \omega_{\wS})\oplus H^2(\wA_X, \mu^*(a_{X*}\omega_X)),
		\end{eqnarray*}
		so $q_{\tilde f} := q(\tilde X) - q(\wS) = h^2(\wA_X, \mu^*(a_{X*}\omega_X))$.
		Since $a_{Y*}\cM$ is a direct summand of $a_{X*}\omega_X$, $\mu^*(a_{Y*}\cM)$ is a direct sum of $\mu^*(a_{X*}\omega_X)$ and we conclude that
		\[
		h^2(\wA_X, \mu^*(a_{X*}\omega_X))\geq h^2(\wA_X, \mu^*(a_{Y*}\cM))=\mathrm{rk}( a_{Y*}\cM)=(\deg a_{S})(g(C)-g(D)),
		\]
		where the second equality is because $\mu^*(a_{Y*}\cM)$ is a direct sum of copies of $\mo_{\tilde A_X}$ by construction.
		
		Let $B$ be a general hyperplane section of $\wS$ and $\tilde X_B=\tilde f^{-1}(B)$. Let $\tilde f_B\colon \tilde X_B \rightarrow B$ be the restriction of $\tilde f$ over $B$.
		\begin{claim}\label{clm: equal H1}
			$h^0(B, R^1\tilde f_{B*}\cO_{\tilde X_B})  = h^0(\wS, R^1{\tilde f_*} \cO_{\wX}) $
		\end{claim}
		\begin{proof}[Proof of the claim.]
			Applying ${\tilde f_*}$ to the short exact sequence of sheaves on $\tilde X$:
			\[
			0\rightarrow \mo_{\tilde X}(-\tilde f^*B) \rightarrow \mo_{\tilde X} \rightarrow \mo_{\tilde X_B} \rightarrow 0
			\]
			we obtain an exact sequence of sheaves on $\tilde S$:
			\[
			f_* \mo_{\tilde X_B} \rightarrow (R^1 {\tilde f_*} \mo_{\tilde X})(-B) \rightarrow  R^1 {\tilde f_*} \mo_{\tilde X} \rightarrow R^1 {\tilde f}_{B*}\mo_{\tilde X_B} \rightarrow  (R^2 {\tilde f_*} \mo_{\tilde X})(-B)
			\]
			For a general hyperplane $B\subset \wS$, one deduces
			\[
			0\rightarrow (R^1 {\tilde f_*} \mo_{\tilde X})(-B) \rightarrow  R^1 {\tilde f_*} \mo_{\tilde X} \rightarrow R^1 {\tilde f}_{B*}\mo_{\tilde X_B} \rightarrow 0
			\]
			Since we have $H^0(\wS, (R^1 {\tilde f_*} \mo_{\tilde X})(-B))  =0$ and $H^1( \wS, (R^1 {\tilde f_*} \mo_{\tilde X})(-B)) = 0$ for a sufficiently ample hypersurface $B$ on $\tilde S$, the induced exact sequence of cohomology groups gives the desired isomorphism:
			\[
			H^0(\wS,R^1 {\tilde f_*} \mo_{\tilde X}) \cong H^0(B, R^1 \tilde f_{B*}\mo_{\tilde X_B}).
			\]
		\end{proof}
		By Claim~\ref{clm: equal H1}, we obtain
		\begin{equation}\label{eq: q_f}
		q_{\tilde f_B} = h^0(B, R^1\tilde f_{B*}\cO_{\tilde X_B})  = h^0(\wS, R^1{\tilde f_*} \cO_{\wX}) = q_{\tilde f}\geq (\deg a_S)(g(C)-g(D)).
		\end{equation}
		
		Suppose that $g(D)=1$. If $g(C)\geq 8$ then $q_{\tilde f_B}\geq g(C)-1\geq \frac{5g(C)+1}{6}$, and hence $\tilde f_B$ is birationally trivial by Xiao's result \cite{X87}. Since $B\subset \tilde S$ is a general hypersurface, $\pi_1(B)\rightarrow \pi_1(\tilde S)$ is surjective and the birational triviality of $\tilde f_B\colon \tilde X_B\rightarrow B$ implies that of $\tilde f\colon \tilde X\rightarrow \tilde S$. Therefore, $\tilde{X}$ is birational to ${\tilde S}\times C$ and $\tilde{Y}$ is birational to $\tilde S\times D$. By the fact that $$\chi(\wX, \omega_{\wX})=(\deg \mu) \chi(X, \omega_X)=(\deg\mu )\chi(Y, \omega_Y)=\chi(\wY, \omega_{\wY}),$$ we infer that $C=D$, which is absurd. Thus in this case $g(C)\leq 7$ and hence $|G|\leq |\mathrm{Aut}(C)|\leq 84\times (7-1)=504$, where the second inequality is Hurwitz's bound of the full automorphism group of a curve.
		
		If $g(D)\geq 2$, then by \eqref{eq: RH} and \eqref{eq: q_f}
		\[
		q_{\tilde f_B} \geq (\deg a_S)(g(C)-g(D)) = \left(\deg a_S\right)\left(\frac{\delta}{2}+(|G|-1)(g(D)-1)\right).
		\]
		where $\delta$ is as in \eqref{eq: RH}. By a simple calculation, if $g(C)\geq 8$ and $|G|\geq 13$, then $q_{\tilde f_B}>\frac{5g(C)+1}{6}$ and by Xiao's result \cite{X87}, $\tilde f_B$ and hence $\tilde f$ is birationally trivial, and we draw a contradiction by the same argument as in the last paragraph. Thus either $|G|\leq 12$ or $g(C)\leq 7$, in which case $|G|\leq \frac{g(C)-1}{g(D)-1} \leq 6$.
	\end{proof}
	
	\section{Threefolds of general type with Gorenstein minimal models}
	\begin{thm}\label{thm: pg}
		There is a constant $M$ such that the following holds. Let $X$ be a projective threefold of general type with canonical singularities such that its minimal models are Gorenstein. Then $|\Aut_\mo(X)|<M$.
	\end{thm}
	
	\begin{proof}
		Since $X$ has a  Gorenstein minimal model, the Miyaoka--Yau inequality gives
		\begin{multline}\label{eq: MY}
		\vol(K_X)\leq 72\chi(X, \omega_X) = 72(p_g(X) - h^1(X,\omega_X) + q(X) -1) \\
		\leq 72(p_g(X) +q(X) -1 )
		\end{multline}
		In particular, $\chi(X, \omega_X)>0$.
		
		By Lemma~\ref{lem: chi 0}, $\Aut_\mo(X) = \Aut_\mo(X)\cap \Aut_A(X)$. If $q(X)\geq 3$ then $|\Aut_\mo(X)|$ is uniformly bounded by Theorem~\ref{thm: mAd}.
		
		In the following we can assume that $q(X)\leq 2$. Let $Y=X/\Aut_\mo(X)$ be the quotient variety and $\pi\colon X\rightarrow Y$ the quotient morphism. Then there is an effective divisor $B$, which is supported on the branched divisor of $\pi$ and with coefficients from the set $\mc_2:=\{1-\frac{1}{n}\mid n\in \ZZ_{>0}\}$, such that $K_X= \pi^*(K_Y+\Delta)$. Obviously, $K_Y+\Delta$ is big and, by \cite[Proposition~5.20]{KM98}, the pair $(Y, \Delta)$ is klt. Since $\Aut_\mo(X)$ acts trivially on $H^0(X, K_X)$, one has $p_g(X)  = p_g(Y, \Delta)$.
		
		By Theorem~\ref{thm: Noether}, there are some positive constants $a$ and $b$ such that
		\begin{equation}\label{eq: Noether}
		\vol(K_X)=|\Aut_\mo(X)|\vol(K_Y+\Delta) \geq |\Aut_\mo(X)|(ap_g(X) -b).
		\end{equation}
		Combining the inequalities \eqref{eq: MY} and \eqref{eq: Noether}, we obtain
		\[
		|\Aut_\mo(X)|(ap_g(X) -b) \leq  72(p_g(X) +1 ).
		\]
		Now one sees easily that, there is a sufficiently large integer $N$ such that $|\Aut_\mo(X)|\leq \lceil \frac{72}{a}+1\rceil$ holds if $p_g(X)>N$. On the other hand, if $p_g(X)\leq N$ then, by \eqref{eq: MY}, we have $\vol(K_X)\leq 72(N+1)$, which implies that this class of threefolds of general type are bounded. It follows that there is a constant $M'>0$ such that $|\Aut(X)|\leq M'$, provided $p_g(X)\leq N$.
		
		Then $M=\max\{\lceil\frac{72}{a}+1\rceil, M'\}$ is the bound we wanted.
	\end{proof}
	
	\begin{cor}\label{cor: Aut Q}
		Let $X$ be a projective threefold of general type with canonical singularities such that its minimal models are Gorenstein. Then $|\Aut_\QQ(X)|\leq M$, where $M$ is as in Theorem~\ref{thm: pg}.
	\end{cor}
	\begin{proof}
		By Lemma~\ref{lem: Hodge}, $\Aut_\QQ(X)$ is contained in the group $\Aut_\mo(X)$.
	\end{proof}
	
	Recall that a smooth projective variety is isogenous to a product of curves if it admits the product of smooth projective curves as an \'etale cover.
	\begin{cor}\label{cor: isog}
		Let $X$ be a projective threefold of general type isogenous to a product of curves. Then $|\Aut_\QQ(X)|\leq M$, where $M$ is as in Theorem~\ref{thm: pg}.
	\end{cor}
	\begin{proof}
		Since $X$ is isogenous to a product of curves, it is smooth with ample $K_X$. Now apply Corollary~\ref{cor: Aut Q}.
	\end{proof}

	\appendix
	
	\section{Noether type inequalities for log canonical pairs of general type}
	
	Let $\mc\subset(0,1]$ be a subset, possibly empty.
	Define $\fP_n (\mc)$ to be the set of $n$-dimensional projective log canonical pairs $(X, \Delta)$ such that $\mc_\Delta\subset\mc$ and $K_X+\Delta$ is big, where $\mc_\Delta$ denotes the set of nonzero coefficients appearing in $\Delta$. Furthermore, set
	\[
	\fP_n^+ (\mc):=\{(X, \Delta) \in\fP_n (\mc) \mid p_g(X, \Delta)>0\},
	\]
	where $p_g(X, \Delta):=h^0(X, K_X+\lfloor \Delta \rfloor)$, and
	\[
	\fV_n^+(\mc):=\{\vol(K_X+\Delta) \mid (X, \Delta)\in \fP_n^+(\mc)\} \text{ and }v_n^+(\mc): =\inf\fV_n^+(\mc).
	\]
	
	\begin{rmk}
		If $\mc$ satisfies the descending chain condition, then  the set
		\[
		\fV_n(\mc):=\{\vol(K_X+\Delta) \mid (X, \Delta)\in \fP_n(\mc)\}
		\]
		also satisfies the descending chain condition (\cite[Theorem~1.3]{HMX14}).  In particular, any nonempty subset of $\fV_n(\mc)$ attains its minimum.
	\end{rmk}
	{
		Now we state the main result of this appendix.
		\begin{thm}\label{thm: Noether}
			Let $n$ be a positive integer and $\mc\subset (0,1]$ a subset such that $\mc\cup\{1\}$ attains the minimum, denoted by $c$. Then there exist positive numbers $a_n(\mc)$ and $b_n(\mc)$, depending on $n$ and $c$, such that for any $(X, \Delta)\in \fP_n(\mc)$ the following inequality holds
			\[
			\vol(K_X +\Delta) \geq a_n(\mc) p_g(X, \Delta) -b_n(\mc)
			\]
		\end{thm}
		
		We make some preparation before proving Theorem~\ref{thm: Noether}.  First we prove a lemma on the extension of pluri-log-canonical sections:
		\begin{lem}\label{lem: ext}
			Let $(X, \Delta + H)$ be a divisorial log terminal pair with $\QQ$-coefficients such that $\lfloor\Delta\rfloor =0$ and $H$ is a reduced smooth divisor.  Suppose that $K_X+ \Delta + H$ is big and no component of $H$ is contained in the stable base locus of $K_X+ \Delta + H$.
			Then the natural restriction map
			\[
			r_m\colon H^0(X, m(K_X+\Delta +H)) \rightarrow H^0(H, m(K_H +\Delta_H ))
			\]
			is surjective for any sufficiently divisible positive integer $m\geq 2$, where $\Delta_H = \Delta|_H$ be restriction of $\Delta$ to $H$.
		\end{lem}
		\begin{proof}
			The proof is similar to the one sketched in \cite[Remark~2.5]{CJ17}.   Since $K_X+ \Delta + H$ is big and $H$ is not contained in the stable base locus of $K_X+ \Delta + H$,  there are an ample $\QQ$-divisor $A$ and an effective $\QQ$-divisor $B$ such that $K_X+\Delta+H\sim_\QQ A+B$ and $H$ is not contained in the support of $B$.  For $0<\epsilon_1\ll \epsilon_2\ll 1$, we may choose an effective $\QQ$-divisor
			\[
			D \sim_\QQ (1-\epsilon_1)(\Delta+ H) + (\epsilon_1 (\Delta+ H) +\epsilon_2 A) + \epsilon_2 B
			\]
			so that $(X, D)$ is klt. By \cite{BCHM10}, there is a minimal model program yielding a birational contraction $\pi\colon (X, D) \dashrightarrow (\bar X, \bar D)$ onto a minimal model $(\bar X, \bar D)$ of $(X, D)$.
			
			Now observe that $K_X + D \sim_\QQ (1+\epsilon_2)(K_X+\Delta + H)$, so $\pi$ yields also a (divisorial log terminal) minimal model $(\bar X, \bar\Delta + \bar H)$ of $(X, \Delta+H)$, where $\bar \Delta = \pi_*\Delta$ and $\bar H =\pi_*H$; see \cite[3.31]{KM98}.
			
			We claim that no component of $H$ is contracted by $\pi$. Otherwise, suppose that a component $H_i$ of $H$ is contracted by $\pi$. Then the discrepancy of $H_i$ with respect to $(\bar X, \bar\Delta+\bar H)$ is larger than $-1$. It follows that $H_i$ is contained in the stable base locus of $K_X+\Delta+H$, which is a contradiction to the assumption.
			
			The Kawamata--Viehweg vanishing theorem, applied to the big and nef divisor $(m-1)(K_{\bar X} +\bar\Delta +\bar H)$ on the Kawamata log terminal pair $(\bar X,\bar\Delta)$, gives the vanishing $H^1(X,  m(K_{\bar X}+\bar\Delta+\bar H)-\bar H)=0$ for any positive integer $m\geq 2$ such that $m(K_{\bar X}+\bar\Delta+\bar H)$ is Cartier. By the long exact sequence of cohomology groups, the natural restriction map $\bar r_m\colon H^0(\bar X, m(K_{\bar X}+\bar\Delta +\bar H)) \rightarrow H^0(\bar H, m(K_{\bar H} +\bar \Delta_{\bar H}))$ is surjective, where $\bar \Delta_{\bar H} := \bar \Delta|_{\bar H}$.
			
			By going to a common log resolution of $(X, \Delta+H)$ and $(\bar X, \bar \Delta+\bar H)$, one sees easily that there is a commutative diagram
			\[
			\begin{tikzcd}
			H^0(\bar X, m(K_{\bar X}+\bar\Delta +\bar H)) \arrow[r, "\bar r_m"]  \arrow[d, "\cong"'] &  H^0(\bar H, m(K_{\bar H} +\bar \Delta_{\bar H})) \arrow[d, "\cong"]\\
			H^0(X, m(K_X+\Delta +H)) \arrow[r, " r_m"]  &  H^0(H, m(K_H +\Delta_H ))
			\end{tikzcd}
			\]
			where the vertical arrows are isomorphisms induced by $\pi\colon X\dashrightarrow \bar X$ and $\pi|_H\colon H\dashrightarrow \bar H$. Now the surjectivity of $r_m$ follows from that of $\bar r_m$.
		\end{proof}
		
		\begin{lem}\label{lem: ind}
			Let $(X,\Delta)$ be a projective log canonical pair of dimension $n$ such that $X$ is smooth and $K_X+\Delta$ is big. Suppose that the linear system $|K_X+\lfloor \Delta \rfloor |$ has positive dimension and its movable part $\Mov|K_X+\lfloor \Delta \rfloor |$ is base point free. Let $H\in \Mov|K_X+\lfloor \Delta \rfloor |$ be a general element, and $\Delta_H := \Delta|_H$ the restriction of $\Delta$ to $H$. Then
			\[
			\vol(K_X+\Delta) \geq 2^{1-n}\vol(K_H +\Delta_H ).
			\]
		\end{lem}
		\begin{proof}
			Since $\Mov|K_X+\lfloor \Delta \rfloor |$ is a positive dimensional base-point-free linear system, its general element $H$ is a (nonzero) smooth divisor by the Bertini theorem.
			
			Since $(X,\Delta)$ is a projective log canonical pair such that $X$ is smooth and $K_X+\Delta$ is big, we can decrease the coefficients of $\Delta$ slightly to obtain a $\QQ$-divisor $\Delta^{(s)} $  for each positive integer $s$ such that the following holds.
			\begin{enumerate}
				\item  $(K_X, \Delta^{(s)} + H)$ is a divisorial log terminal pair with $\QQ$-coefficients such that $\lfloor\Delta^{(s)} \rfloor =0$.
				\item $ K_X+\Delta^{(s)}$ is big, and  $K_X+\Delta^{(s)} - (1-\frac{1}{s})H$ is $\QQ$-linearly equivalent to an effective $\QQ$-divisor.
				\item $\lim_{s\rightarrow \infty} \Delta^{(s)} = \Delta$.
			\end{enumerate}
			By \cite{BCHM10}, there is a minimal model program yielding a birational contraction 
			\[
			\pi_s\colon (X, \Delta^{(s)} +H)\dashrightarrow (X_s, \Delta_s + H_s),
			\]
			where $\Delta_s = \pi_{s*}\Delta^{(s)}$ and $H_s=\pi_{s*} H$, such that $ (X_s, \Delta_s+H_s)$ is a minimal model of $(X, \Delta^{(s)}+H)$ with divisorial log terminal singularities. Resolving the  indeterminacy of  $\pi_s$, we obtain the following commutative diagram of birational maps:
			\[
			\begin{tikzcd}
			& \tilde X_{s} \arrow[ld, "p_s"'] \arrow[rd, "q_s"]& \\
			X \arrow[rr, dashed, "\pi_s"] && X_s
			\end{tikzcd}
			\]
			where $p_s$ and $q_s$ are log resolutions of $(X, \Delta^{(s)} +H )$ and $(X_s, \Delta_s +H _s)$ respectively.
			
			Denote by $\tilde H_s $ be the strict transform of $H$ on $\tilde X_s$. There is a uniquely determined divisor $\tilde \Delta_s$ on $\tilde X_s$ such that $p_s^*(K_X+\Delta^{(s)} + H) = K_{\tilde X_s}+\tilde \Delta_s + \tilde H_s $ and $p_{s*}\tilde \Delta_s =\Delta^{(s)}$.  Let $\Delta_{\tilde H_s}$ be the restriction of $\tilde \Delta_s^{>0}$ to $\tilde H_s $, and $\Delta_H^{(s)}$ the restriction of $\Delta^{(s)}$ to $H$. We claim that
			\begin{equation}\label{eq: equal vol}
			\vol\left(K_H+\Delta_H^{(s)}\right) = \vol(K_{\tilde H_s}+ \Delta_{\tilde H_s}).
			\end{equation}
			In fact, since $p_{s*}\Delta_{\tilde H_s} = \Delta_H^{(s)}$, one has $\vol(K_H+\Delta_H^{(s)}) \geq \vol(K_{\tilde H_s}+ \Delta_{\tilde H_s})$. On the other hand,
			\begin{multline*}
			p_s^*(K_H+\Delta_H^{(s)}) = p_s^*(K_X+H+\Delta^{(s)})|_{\tilde H_s} = (K_{\tilde X_s}+\tilde H_s+\tilde \Delta_s)|_{\tilde H_s} \\
			\leq  (K_{\tilde X_s}+\tilde H_s+\tilde \Delta_s^{>0})|_{\tilde H_s} = K_{\tilde H_s}+ \Delta_{\tilde H_s}
			\end{multline*}
			and thus $\vol(K_H+\Delta_H^{(s)})  = \vol(p_s^*(K_H+\Delta_H^{(s)}))\leq \vol(K_{\tilde H_s}+ \Delta_{\tilde H_s})$.
			
			For sufficiently large and divisible integer $m$, consider the commutative diagram
			\begin{equation}\label{eq: surj}
			\begin{tikzcd}
			H^0(\tilde X_s, mq_s^*(K_{X_s}+\Delta_s +H_s)) \arrow[r, "r_{s,m}'"] \arrow[d, "\cong"', "\varphi_{s,m}"]&  H^0(\tilde H_s, mq_s^*(K_{X_s}+\Delta_s +H_s)|_{\tilde H_s}) \arrow[d, hook, "\lambda_{s,m}"]\\
			H^0(\tilde X_s, m(K_{\tilde X_s}+\tilde \Delta_s^{>0} + \tilde H_s)) \arrow[r, twoheadrightarrow, " r_{s,m}"] &  H^0(\tilde H_s, m(K_{\tilde X_s}+\tilde \Delta_s^{>0} + \tilde H_s)|_{\tilde H_s}) \\
			\end{tikzcd}
			\end{equation}
			where the horizontal arrows are restriction maps and the vertical maps are induced by inclusions of sheaves and hence injective. Since $(X_s, \Delta_s +H_s)$ is a minimal model of $(\tilde X_s, \tilde \Delta_s^{>0} + \tilde H_s)$, the map $\varphi_{s,m}$ is an isomorphism; the restriction map $r_{s,m}$ is surjective by Lemma~\ref{lem: ext}. It follows that $\lambda_{s,m}$ is surjective and hence an isomorphism. Letting $m$ go to infinity, we obtain
			\begin{multline}\label{eq: vol K_H}
			\vol(K_{\tilde H_s} + \Delta_{\tilde H_s}) = \vol((K_{\tilde X_s}+\tilde \Delta_s^{>0} + \tilde H_s)|_{\tilde H_s}) \\
			= \vol (q_s^*(K_{X_s}+\Delta_s +H_s)|_{\tilde H_s}) = q_s^*(K_{X_s}+\Delta_s +H_s)^{n-1}\cdot \tilde H_s
			\end{multline}
			where the second equality is due to the isomorphisms $\lambda_{s,m}$ and the last equality is because of the nefness of $q_s^*(K_{X_s}+\Delta_s +H_s)|_{\tilde H_s}$.
			
			Since $K_X+\Delta \geq H$ and $K_X+\Delta \geq K_{X} +\Delta^{(s)}$ by construction, we have
			\begin{equation}\label{eq: 2K}
			2(K_X+\Delta)  \geq K_{X} +\Delta^{(s)} + H
			\end{equation}
			Also, since $K_{X} +\Delta^{(s)}\geq (1-\frac{1}{s}) H$, we have
			\begin{equation}\label{eq: K+H}
			K_{X_s} + \Delta_s + H_s = \pi_{s*} (K_X+\Delta^{(s)}+H) \geq \left(2-\frac{1}{s}\right) \pi_{s*} H = \left(2-\frac{1}{s}\right) H_s.
			\end{equation}
			Now we can compute the volume
			\begin{align*}
			2^n\vol(K_X+\Delta) &=\vol(2(K_X+\Delta)) \\
			&\geq \vol(K_X+\Delta^{(s)}+H) \hspace{.5cm}\text{(by \eqref{eq: 2K})}\\
			& =\vol(K_{X_s}+\Delta_s+ H_s) \\
			&=(K_{X_s}+\Delta_s+ H_s)^n  \hspace{.5cm}\text{(since $K_{X_s}+\Delta_s +H_s$ is nef)}\\
			&\geq (K_{X_s}+\Delta_s+ H_s)^{n-1}\cdot \left(2-\frac{1}{s}\right) H_s \hspace{.5cm}\text{(by \eqref{eq: K+H})} \\
			&= \left(2-\frac{1}{s}\right)  (q_s^*(K_{X_s}+\Delta_s +H_s))^{n-1}\cdot q_s^*H_s  \\
			& \geq \left(2-\frac{1}{s}\right) (q_s^*(K_{X_s}+\Delta_s +H_s))^{n-1}\cdot\tilde H_s \hspace{.5cm}\text{(since $q_s^*H_s \geq \tilde H_s$)}\\
			&  = \left(2-\frac{1}{s}\right)\vol(K_{\tilde H_s}  +\Delta_{\tilde H_s} ) \hspace{.5cm}\text{(by \eqref{eq: vol K_H})} \\
			&= \left(2-\frac{1}{s}\right)\vol(K_H +\Delta_H^{(s)} ) \hspace{.5cm}\text{(by \eqref{eq: equal vol})}
			\end{align*}
			Since $\lim_{s\rightarrow\infty} \Delta^{(s)}_H=\Delta_H$, we obtain
			\[
			\vol(K_X+\Delta)  \geq 2^{-n}\lim_{s\rightarrow\infty} \left(2-\frac{1}{s}\right)\vol\left(K_H +\Delta_H^{(s)} \right)  \geq 2^{1-n}\vol(K_H +\Delta_H).
			\]
		\end{proof}

		\begin{proof}[Proof of Theorem~\ref{thm: Noether}]
			By ensuring that $b_n(\mc)\geq a_n(\mc)$, the inequality of the theorem is trivially true in the case $p_g(X, \Delta)\leq 1$. We can thus assume that $p_g(X, \Delta)\geq 2$.
			
			We proceed by induction on the dimension $n$.
			
			If $n=1$ then $X$ is a smooth projective curve, and, using the Riemann--Roch theorem,
			\[
			\vol(K_X +\Delta) = \deg (K_X +\Delta) \geq \deg (K_X +\lfloor\Delta\rfloor) \geq h^0(X, K_X +\lfloor\Delta\rfloor) -1
			\]
			where the last inequality is an equality if and only if $g(X)=0$. Thus we can take $a_1(\mc) =b_1(\mc) = 1$ for any $\mc\subset (0,1]$.
			
			From now on, we assume that $n\geq 2$ and that the theorem has been proven in lower dimensions. Let $(X, \Delta)$ be an $n$-dimensional projective log canonical pair of general type with coefficients of $\Delta$ belonging to $\mc$.
			
			First we assume that $1\in \mc$. Let $\rho\colon \tilde X\rightarrow X$ be a log resolution of $(X, \Delta)$ such that the movable part $\Mov|K_{\tilde X}+\lfloor\tilde \Delta^{>0}\rfloor|$ of the linear system $|K_{\tilde X}+\lfloor\tilde \Delta^{>0}\rfloor|$ is base point free, where $\tilde \Delta$ is the divisor on $\tilde X$ satisfying $K_{\tilde X} + \tilde \Delta = \rho^*(K_X+\Delta)$ and $\rho_*\tilde\Delta =\Delta$, and $\tilde \Delta^{>0}$ denotes the effective part of $\tilde \Delta$. Let $\tilde \Delta'$ be the divisor on $\tilde X$, obtained from $\tilde \Delta^{>0}$ by raising the coefficients not in $\mc$ to 1 and then decreasing those coefficients in $\mc\cap(0,1)$ (if any) to $c=\min (\mc\cup\{1\})$. Then $(\tilde X, \tilde \Delta')\in\fP_n(\{c,1\})$, and  one sees easily that
			\[
			0<\vol(K_{\tilde X}+ \tilde \Delta') \leq \vol(K_X+\Delta) \text{ and } p_g(\tilde X, \tilde \Delta') = p_g(X, \Delta).
			\]
			Replacing $(X, \Delta)$ with $(\tilde X, \tilde \Delta')$, we can assume that $X$ is smooth, $\Delta$ is a simple normal crossing divisor with coefficients lying in $\{c,1\}$, and $\Mov|K_X+\lfloor \Delta\rfloor|$ is base point free.
			
			Let $H\in \Mov|K_X+\lfloor \Delta\rfloor|$ be a general member. Then $H$ is smooth and $H+\Delta$ is simple normal crossing by the Bertini theorem. Let $\Delta_H$ be the restriction of $\Delta$ to $H$.

			Let $f\colon X\rightarrow \PP^{p_g-1}$ denote the morphism defined by $| K_X+\lfloor \Delta\rfloor|$, where $p_g = p_g(X, \Delta)$.
			In order to bound $\vol(K_X +\Delta)$ from below in terms of $p_g(X, \Delta)$, we have to distinguish two cases depending on $\dim f(X)$.
			
			\medskip
			
			\noindent{\bf Case $\dim f(X)\geq 2$.} In this case, $H$ is connected. Then $(H, \Delta_H)\in\fP_{n-1}(\{c,1\})\subset \fP_{n-1}(\mc)$, and hence by induction there are constants $a_{n-1}(\mc)$ and $b_{n-1}(\mc)$ such that
			\begin{equation}\label{eq: n-1}
			\vol(K_H+\Delta_H)\geq a_{n-1}(\mc)p_g(H, \Delta_H) - b_{n-1}(\mc)
			\end{equation}
			Since $p_g(H, \Delta_H)\geq p_g(X, \Delta) -1$, we have by Lemma~\ref{lem: ind} and \eqref{eq: n-1}:
			\begin{equation}\label{eq: case1}
			\begin{split}
			\vol(K_X+\Delta)  &\geq 2^{1-n}\vol(K_H+\Delta_H) \\
			&\geq 2^{1-n}\left(a_{n-1}(\mc)(p_g(X, \Delta)-1) - b_{n-1}(\mc)\right) \\
			& = 2^{1-n}a_{n-1}(\mc)p_g(X, \Delta) -  2^{1-n}\left(a_{n-1}(\mc) + b_{n-1}(\mc)\right)
			\end{split}
			\end{equation}
			\medskip
			
			\noindent{\bf Case $\dim f(X)=1$.} In this case, the number $k$ of connected components of $H$ is $\deg f(X)\geq p_g(X, \Delta) -1$. Let $H_1, \dots, H_k$ be the connected components of $H$. For each $1\leq i\leq k$,  let $\Delta_{H_i}$ be the restriction of $\Delta$ to $H_i$. Then $(H_i, \Delta_{H_i})\in\fP_{n-1}^+(\{c,1\})$, and $\vol(K_{H_i}+\Delta_{H_i})$ is independent of $i$. We have thus
			\begin{equation}
			\vol(K_H +\Delta_H) = k\vol(K_{H_i}+\Delta_{H_i}) \geq v_{n-1}^+(\{c,1\}) (p_g(X, \Delta) -1).
			\end{equation}
			Combining this with Lemma~\ref{lem: ind}, we obtain
			\begin{equation}\label{eq: case2}
			\vol(K_X+\Delta) \geq 2^{1-n}v_{n-1}^+(\{c,1\}) (p_g(X, \Delta) -1).
			\end{equation}
			
			By \eqref{eq: case1} and $\eqref{eq: case2}$, if we set
			\begin{align*}
			a_n(\mc) &= 2^{1-n} \min\{a_{n-1}(\mc), v_{n-1}^+(\{c,1\})  \}, \\
			b_n(\mc) &=2^{1-n}\max\{a_{n-1}(\mc) + b_{n-1}(\mc),  v_{n-1}^+(\{c,1\}) \}.
			\end{align*}
			then the inequality $\vol(K_X +\Delta) \geq a_n(\mc) p_g(X, \Delta) -b_n(\mc)$ is valid in both cases.
			
			In general, if $1\notin \mc$, since $\fP_n(\mc)\subset \fP_n(\mc\cup\{1\})$, we can simply set $a_n(\mc):=a_n(\mc\cup\{1\})$ and $b_n(\mc):=b_n(\mc\cup\{1\})$.
		\end{proof}
		\begin{rmk}
			The choices of $a_n(\mc)$ and $b_n(\mc)$ in the proof of Theorem~\ref{thm: Noether} are not optimal, but are computable. If the following Question~\ref{qu: min vol} admits a positive answer, then one can see by induction that, for $n\geq 2$,
			\[
			a_n(\mc)= \frac{1}{2^{n-1}}v_{n-1}^+(\{c,1\})
			\]
			 and 
			 \[
			b_n = a_{n-1} + a_{n-2} +\cdots + a_1+b_1 = 2+\sum_{1\leq i\leq n-1}\frac{1}{2^{i-1}}v_{i-1}^+(\{c,1\}).
			\]
		\end{rmk}
		
		\begin{qu}\label{qu: min vol}
			For $n\geq 2$ and $c\in (0,1]$, is it true that $v_{n}^+(\{c,1\}) \leq  \frac{1}{2^{n-1}}v_{n-1}^+(\{c,1\})$?
		\end{qu}
		
\section*{Acknowledgements}
	Part of this work was done when the first author was visiting the Tianyuan Mathematical Center in Southeast China, located at Xiamen University. We would like to thank the TMCSC and Xiamen University for the support and hospitality. We are grateful to the referee for  careful reading and suggestions. The first author was supported by National Key Research and Development Program of China (No.~2020YFA0713200), by the NSFC for Innovative Research Groups (No.~12121001), by the Natural Science Foundation of Shanghai
	(No.~21ZR1404500), and by the NSFC (No.~11871155 and  No.~11731004). The second author was supported by the NSFC (No.~11971399 and No.~11771294) and by the Presidential Research Fund of Xiamen University (No.~20720210006).

\bibliographystyle{alpha}
  \bibliography{arxiv}			
	\end{document}